\documentclass[11pt]{amsart}
\usepackage[T1]{fontenc}
\usepackage[utf8]{inputenc}

\usepackage{csquotes}
\usepackage[english]{babel}
\usepackage[backend=biber,style=alphabetic]{biblatex}
\DeclareSourcemap{
  \maps[datatype=bibtex]{
    \map{
      \step[fieldsource=doi, final]
      \step[fieldset=url, null]
      \step[fieldset=urldate, null]
    }
  }
}
\usepackage{etoolbox}
\makeatletter
\patchcmd{\@setauthors}
  {\MakeUppercase{\authors}}
  {\authors}
  {}{\fail}
\patchcmd{\maketitle}
  {\@nx\MakeUppercase{\the\toks@}}
  {\the\toks@}
  {}{\fail}
\makeatother

\usepackage{amsmath,amssymb,amsthm}
\usepackage{mathtools}
\usepackage{marginnote}
\usepackage{subcaption}
\usepackage{enumitem}
\setlist[enumerate,1]{label=(\roman*)}

\usepackage[
  a4paper,
  margin=1in,
  headheight=1cm,
  headsep=1cm,
  footskip=1cm,
  marginparwidth=0.85in,
  marginparsep=0.1in
]{geometry}

\usepackage{xcolor}
\definecolor{darkblue}{RGB}{0,0,128}
\definecolor{definitionblue}{RGB}{35,55,95}
\definecolor{objectlinkgray}{RGB}{50,50,50}
\definecolor{commentblue}{rgb}{0.00,0.20,0.70}
\definecolor{commentred}{rgb}{0.65,0.05,0.05}
\definecolor{commentgreen}{rgb}{0.00,0.45,0.20}
\definecolor{commentpurple}{rgb}{0.40,0.10,0.60}
\definecolor{commentorange}{rgb}{0.75,0.30,0.00}
\definecolor{commentteal}{rgb}{0.00,0.45,0.50}
\definecolor{commentmagenta}{rgb}{0.60,0.00,0.40}
\definecolor{commentbrown}{rgb}{0.45,0.25,0.05}
\definecolor{commentolive}{rgb}{0.35,0.45,0.00}
\definecolor{commentgray}{rgb}{0.30,0.30,0.30}

\colorlet{improvementcolor}{commentorange}
\colorlet{questioncolor}{commentmagenta}
\colorlet{remarkcolor}{commentblue}
\usepackage[colorlinks=true,linkcolor=darkblue,citecolor=darkblue,urlcolor=darkblue,filecolor=darkblue,pdfborder={0 0 0}]{hyperref}

\theoremstyle{plain}
\newtheorem{theorem}{Theorem}[section]
\newtheorem{lemma}[theorem]{Lemma}
\newtheorem{proposition}[theorem]{Proposition}
\newtheorem{corollary}[theorem]{Corollary}
\newtheorem{conjecture}[theorem]{Conjecture}

\theoremstyle{definition}
\newtheorem{definition}[theorem]{Definition}
\newtheorem{example}[theorem]{Example}

\newtheorem{problem}[theorem]{Problem}

\theoremstyle{remark}
\newtheorem{remark}[theorem]{Remark}
\newtheorem*{notation}{Notation}

\usepackage[nameinlink,noabbrev]{cleveref}
\crefname{theorem}{theorem}{theorems}
\Crefname{theorem}{Theorem}{Theorems}
\crefname{lemma}{lemma}{lemmas}
\Crefname{lemma}{Lemma}{Lemmas}
\crefname{proposition}{proposition}{propositions}
\Crefname{proposition}{Proposition}{Propositions}
\crefname{corollary}{corollary}{corollaries}
\Crefname{corollary}{Corollary}{Corollaries}
\crefname{conjecture}{conjecture}{conjectures}
\Crefname{conjecture}{Conjecture}{Conjectures}
\crefname{observation}{observation}{observations}
\Crefname{observation}{Observation}{Observations}
\crefname{claim}{claim}{claims}
\Crefname{claim}{Claim}{Claims}
\crefname{example}{example}{examples}
\Crefname{example}{Example}{Examples}
\crefname{question}{question}{questions}
\Crefname{question}{Question}{Questions}
\crefname{problem}{problem}{problems}
\Crefname{problem}{Problem}{Problems}
\crefname{remark}{remark}{remarks}
\Crefname{remark}{Remark}{Remarks}
\crefname{section}{section}{sections}
\Crefname{section}{Section}{Sections}
\crefname{equation}{}{}
\Crefname{equation}{}{}
\crefname{enumi}{}{}
\Crefname{enumi}{}{}
\newcommand{\Crefitem}[2]{%
    \hyperref[#2]{\Cref*{#1}\ref*{#2}}%
}

\usepackage{listings}[2013/08/05]
\lstdefinestyle{gap}{
  language=GAP,
  basicstyle=\ttfamily\small,
  columns=fullflexible,
  keepspaces=true,
  showstringspaces=false,
  upquote=true
}
\newcommand{\gap}[1]{\texttt{\detokenize{#1}}}

\newcommand{\rimp}{\Rightarrow}

\newcommand{\rlimp}{\Leftrightarrow}

\newcommand{\bb}[1]{\mathbb{#1}}
\newcommand{\Alg}[1]{\mathbf{#1}}
\newcommand{\Sym}[1]{\Alg{S}_{#1}}
\newcommand{\sym}[1]{S_{#1}}
\newcommand{\Alt}[1]{\Alg{A}_{#1}}
\newcommand{\alt}[1]{A_{#1}}
\newcommand{\Dih}[1]{\Alg{D}_{#1}}
\newcommand{\dih}[1]{D_{#1}}
\newcommand{\Cent}[2]{\operatorname{C}_{#1}\left(#2\right)}
\newcommand{\LMlt}[1]{\operatorname{LMlt\hspace{-0.3ex}\left(#1\right)}}

\newcommand{\Cl}[2]{\operatorname{\Alg{Cl}}_{#1}\hspace{-0.3ex}\left(#2\right)}
\newcommand{\cl}[2]{\operatorname{Cl}_{#1}\hspace{-0.3ex}\left(#2\right)}
\newcommand{\Inn}[1]{\operatorname{Inn}\left(#1\right)}

\newcommand{\ord}[2]{\operatorname{ord}_{#1}\left(#2\right)}
\newcommand{\Gen}[2]{\left<#2\right>_{#1}}
\newcommand{\slant}[2]{{\raisebox{.2em}{$#1$}\left/\raisebox{-.2em}{$#2$}\right.}}
\newcommand{\restr}[1]{\hspace{-0.8ex}\restriction_{#1}}

\begin{document}
\title{Hayashi Property for Conjugation Quandles}
\author{Filip Filipi}
\thanks{Department of Algebra, Faculty of Mathematics and Physics, Charles University, Prague, Czech Republic. ORCID: \url{https://orcid.org/0009-0005-0438-2012}. E-mail: \texttt{filip.filipi@matfyz.cuni.cz}. }

\keywords{conjugacy class; conjugation; group; Hayashi's conjecture; quandle}
\subjclass[2020]{20E34, 20E45, 20N02, 57K12}
\begin{abstract}
We make a comprehensive and accessible description of conjugation quandles and their connectedness. In this context, we build on Kayacan's recent work and find a characterization of Hayashi's conjecture (2013) in terms of a certain centrality-testing condition of groups. We say that a finite conjugacy class $C$ of a group $\mathbf{G}$ is \emph{good} if it satisfies the following: for every $c\in C$ there exists $z\in C$ such that, for each integer $k$, the element $c^k$ is central in the group $\left<C\right>$ if and only if it commutes with $z$. We show that, according to Hayashi's conjecture, a finite conjugacy class of a group $\mathbf{G}$ should be good whenever it generates the entire group~$\mathbf{G}$. In a~more general setting, we prove that conjugacy classes of finite nilpotent, dihedral, and small groups are good. Elaborating on Kayacan's verification of Hayashi's conjecture for conjugation quandles over symmetric and alternating groups, we verify that these two families of groups are also good. Such types of results translate back as a verification of Hayashi's conjecture whenever the corresponding conjugation quandles become connected.
\end{abstract}
\maketitle

\section{Introduction}
An algebraic structure called \emph{quandle} arises naturally in knot theory. The axioms of this structure are based on the Reidemeister moves. Those are a complete collection of knot diagram manipulations needed to decide whether two diagrams represent the same knot. Using a notion of proper knot diagram colorings, this allows quandles to define an invariant useful in distinguishing (oriented) knots by their diagrams. To those interested in the origin of quandles, we recommend reading the introductory section of~\cite{fish} by Fish, Lisitsa, and Stanovský, which explains all this in a comprehensive yet concise manner. More complete basic knowledge and concepts from knot theory can be found in the book Quandles~\cite{elhamdadi} by Elhamdadi and Nelson.

As stated in Section 5.2 of~\cite{ohtsuki}, if two knots can be distinguished by a quandle, then they can also be distinguished by one of its \emph{connected} subquandles. A structural restriction for finite connected quandles has been conjectured by Hayashi in 2013~\cite{hayashi}. This conjecture has been further studied in~\cite{conjecture1,conjecture2,conjecture4,kayacan,conjecture3, conjecture5}, and it was pointed out in~\cite{review} that the results of~\cite{horosevski} also contribute to it. In general, for a family of so-called \emph{conjugation} quandles, this conjecture still remains open.

This work aims to expand on the ideas of Kayacan~\cite{kayacan}. In there, Kayacan used tools provided in~\cite{racks} to study Hayashi's conjecture for conjugation quandles using group-theoretic language. In \Cref{prop:centralizer-property-general}, we generalize this method to also apply to groups with nontrivial centers and we use it to expose the possible group-theoretic consequences of Hayashi's conjecture.

We introduce the following centrality-testing property of finite conjugacy classes, resp. groups.
\begin{definition}
    We say that a finite conjugacy class $C$ of a group $\Alg{G}$ is \emph{good} in $\Alg{G}$ if for every $c\in C$ there is $z\in C$ such that, for every $k\in \bb{Z}$, it holds that\[c^k\text{ is central in the group }\Gen{}{C} \Leftrightarrow c^k \text{ commutes with }z.\]The group $\Alg{G}$ itself is called \emph{good} if all of its finite conjugacy classes are good in $\Alg{G}$.
\end{definition}
In \Cref{thm:hayashi-equivalence}, we show that Hayashi's conjecture (\Cref{conj:hayashi}), when considered for connected conjugation quandles only, is equivalent to the following statement.
\begin{conjecture}\label{conj:groups}
    If $\Alg{G}$ is a group and $C$ is its finite conjugacy class satisfying $\Gen{}{C} = \Alg{G}$, then $C$ is good in $\Alg{G}$.
\end{conjecture}

 All groups $\Alg{G}$ studied in the paper have their conjugacy classes $C$ good, regardless of the assumption $\left<C\right> = \mathbf{G}$. We thus state the following problem.
\begin{problem}[Extended Hayashi's Problem]\phantomsection\label{problem}
    What groups $\Alg{G}$ are good?
\end{problem}

 We prove that the dihedral groups $\Dih{2n}$, finite nilpotent groups, simple groups of order at most $5\,000\,000$, and general groups of order at most $500$ are good. All results of this type translate back to proofs of Hayashi's conjecture for particular conjugation quandles when the generating condition of a conjugacy class is met. Since the subgroups generated by conjugacy classes are necessarily normal, this generating condition is, in particular, satisfied by all nontrivial finite conjugacy classes of simple nonabelian groups.

 For $n\geq 5$, due to the nature of their structure, the goodness of the groups $\Sym{n}$ and $\Alt{n}$ would follow directly from Kayacan's arguments made in Theorem~7 from~\cite{kayacan}. In this theorem, using the method our article builds on, Kayacan proves that connected conjugation quandles derived from the groups $\Sym{n}$ and $\Alt{n}$ satisfy Hayashi's conjecture. Unfortunately, Kayacan's proof contains a non sequitur claim to which we provide a counterexample in \Cref{sec:special}. To fill in this gap from~\cite[Theorem~7]{kayacan}, we prove \Cref{prop:z-in-sym-and-alt}, reassuring us that the groups $\Sym{n}$ and $\Alt{n}$ are indeed good.
 
\section{Preliminaries}
Most of the terminology used is based on~\cite{cvrcek}.

\begin{notation}
    Let there be a set $Q$. Given a subgroup $\Alg{H}$ of the symmetric group $\Sym{Q}$, we denote $\Alg{H}_z$, the \emph{stabilizer subgroup} of $\Alg{H}$ with respect to $z\in Q$, that is,  $H_z= \{\pi\in \Sym{Q}\mid  \pi(z) =z\}$.
\end{notation}

\begin{notation}
    Given a group $\Alg{G}$ and its element $g$, by $\phi_g$ we denote the (left) \emph{conjugation} by an element $g$ in the group $G$, that is, $\phi_g:G\to G, \ x\mapsto gxg^{-1}$. Such mappings are automorphisms of $\Alg{G}$, and the set of all such mappings forms the \emph{inner automorphism group} of $\Alg{G}$, denoted by $\Inn{\Alg{G}}\leq \Sym{G}$. For an element $h$ of a subgroup $\Alg{H}\leq \Alg{G}$, it should be clear from the context whether the domain of $\phi_h$ is considered to be $H$ or $G$. 
\end{notation}

\begin{notation}
    Given a group $\Alg{G}$ and its element $z$, we denote $\Cent{\Alg{G}}{z}\coloneq\{g\in G\mid gz=zg\}\leq \Alg{G}$, the \emph{centralizer} of $z$ in the group $\Alg{G}$.
\end{notation}

\begin{definition}
    Let $Q$ be a finite nonempty set and $\pi \in \sym{Q}$. We say that $(\lambda_1^{\alpha_1}, \dots, \lambda_t^{\alpha_t})$ is a~\emph{cycle structure}\footnote{or \emph{profile}} of a permutation $\pi$ if the following conditions are satisfied
    \begin{itemize}
        \item $\lambda_1< \lambda_2< \dots< \lambda_t$,
        \item for every $i\in \{1, \dots, t\}$, $\alpha_i\in \bb{N}$ is the number of cycles of length $\lambda_i$ in $\pi$,
        \item each cycle is counted, that is $\sum_{i=1}^{t}\alpha_i\cdot\lambda_i = \left|Q\right|$.
    \end{itemize}

    As long as $Q$ is finite, it always exists, is uniquely defined, and is preserved upon conjugation by elements of $S_Q$. Furthermore, if it holds that $\lambda_i\mid\lambda_t$ for all $i\in\{1,\dots, t\}$, we say that the permutation $\pi$ contains a \emph{regular cycle}.
\end{definition}

\begin{remark}
    Whenever $x\in Q$ is an element of a cycle of length $\lambda$ in $\pi\in \sym{Q}$, then $\lambda$ equals the smallest $k\in \bb{N}$ such that $\pi^k(x) = x$.
\end{remark}

\begin{definition}
    Let $Q$ be a nonempty set and $\star$ a binary operation on it. The tuple $\Alg{Q} = (Q, \star)$ is then called a \emph{quandle} if for each $a\in Q$ it holds $a\star a = a$ and the \emph{left translation}\footnote{one might encounter dual definition of a quandle using right translations} $$L_a: Q\to Q,\ x\mapsto a\star x$$ is an automorphism of $\Alg{Q}$.

    We denote $\LMlt{\Alg{Q}}\coloneq\Gen{}{L_a\mid a\in Q}\leq \Sym{Q}$ the \emph{(left) multiplication group} of a quandle $\Alg{Q}$ and we say $\Alg{Q}$ is \emph{connected}\footnote{or \emph{indecomposable}} if $\LMlt{\Alg{Q}}$ acts \emph{transitively} on $Q$, i.e., for any $a,b\in Q$ there is $\varphi\in\LMlt{\Alg{Q}}$ such that $b = \varphi (a)$.
\end{definition}

It is useful to introduce the following property of quandles.
\begin{definition}
    We say that a finite quandle $\Alg{Q}$ has the \emph{Hayashi property} if all of its left translations contain a regular cycle.
\end{definition}

Using it, we may now state a conjecture proposed in 2013 in~\cite{hayashi} as follows.
\begin{conjecture}[Hayashi]\phantomsection\label{conj:hayashi}
    Finite connected quandles have the Hayashi property.
\end{conjecture}

Given a quandle automorphism $\alpha$ and $a,x$ elements of its domain, \begin{equation}\label{eq:translations-are-conjugate}
    L_{\alpha(a)}(x) = \alpha(a)\star x = \alpha(a\star \alpha^{-1}(x)) = \alpha\circ L_a\circ \alpha^{-1}(x).
\end{equation}In connected quandles, for any two left translations $L_a, L_b$, there is a quandle automorphism $\varphi$ such that $b  = \varphi(a)$. Due to Eq.~\eqref{eq:translations-are-conjugate}, it holds that $L_b = L_{\varphi(a)} = \varphi\circ L_a \circ \varphi^{-1}$, and so $L_a$ and $L_b$ have the same cycle structure. This shows that a connected quandle has the Hayashi property if and only if at least one of its left translations contains a regular cycle.

\section{Conjugation Quandles and their Connectedness}\label{sec:conjugation-quandles}

We work with a special type of quandle that is related to groups. This shall be the main object studied in this paper.
\begin{definition}
    Let $\Alg{G}$ be a group and $\Alg{C} \coloneq (C, \star)$ a quandle. If $C\subseteq G$ and $a\star b = aba^{-1}$ for each $a,b\in C$, then we say that $\Alg{C}$ is a \emph{conjugation quandle} over the group $\Alg{G}$.
\end{definition}
Description of conjugation quandles seems to be considered folklore but, for completeness, we provide the proofs. Many of the necessary ideas were already presented, for example, in~\cite{kayacan}. The following lemma is a simple observation that will be used with no further mention.
\begin{lemma}\label{lemma:translations}
    Let $\Alg{C}$ be a conjugation quandle over a group $\Alg{G}$. Then\begin{enumerate}
        \item\label{it:translations:1} for any $c\in C$ it holds $L_c = \phi_c\restr{C}$ and $L_c^{-1} = \phi_{c^{-1}}\restr{C}$,
        \item $\LMlt{\Alg{C}} = \Inn{\Gen{}{C}}\restr{C}$ as groups.
    \end{enumerate}
\end{lemma}
\begin{proof}\hfill
    \begin{enumerate}
        \item First part is given by the definition of the quandle operation of $\Alg{C}$. To prove the second part, observe that $$\phi_{c^{-1}}\restr{C} \circ L_c = \phi_{c^{-1}}\restr{C}\circ\  \phi_c\restr{C}\  = \operatorname{id}_C,$$ and since we know $L_c$ is a bijection, $\phi_{c^{-1}}\restr{C}$ has to be its both-sided inverse.
        \item According to \Cref{it:translations:1},$$\LMlt{\Alg{C}} = \Gen{}{\{L_c\mid c\in C\}} = \Gen{}{\{\phi_c\restr{C}\mid c\in C\}}.$$ Since it holds that $\phi_c\restr{C}\circ\  \phi_{c'}\restr{C}\  = \phi_{cc'}\restr{C}$ and $(\phi_c\restr{C})^{-1} = \phi_{c^{-1}}\restr{C}$, we additionally deduce\[\Gen{}{\{\phi_c\restr{C}\mid c\in C\}} = \left\{\phi_h\restr{C}\mid h\in\Gen{}{C}\right\}=\Inn{\Gen{}{C}}\restr{C}.\qedhere\]
    \end{enumerate}
\end{proof}

\begin{theorem}\label{thm:classification}
    Let $\Alg{G}$ be a group and $C\subseteq G$ be nonempty. Then TFAE
    \begin{enumerate}
        \item\label{it:classification:1} there is a conjugation quandle over the group $\Alg{G}$ on $C$,
        \item\label{it:classification:2} $C$ is closed under conjugation by elements of $\Gen{}{C}$,
        \item\label{it:classification:3} $C$ is a union of conjugacy classes of the group $\Gen{}{C}$.
    \end{enumerate}
\end{theorem}
\begin{proof}\hfill

    \begin{enumerate}
        \item[\Cref{it:classification:1}$\rimp$\Cref{it:classification:2}] Let $\Alg{C}\coloneq (C, \star)$ be a conjugation quandle over $\Alg{G}$. For any $c,a\in C$ it holds that $cac^{-1} = c\star a\in C$ and $c^{-1}ac = \phi_{c^{-1}}\restr{C}\hspace{-0.8ex}(a) = L_c^{-1}(a)\in C$, so $C$ is closed under conjugation by elements of $C$ and their inverses. By induction, this can be shown for products, and so it is closed upon conjugation by elements of $\Gen{}{C}$.
        \item[\Cref{it:classification:2}$\rimp$\Cref{it:classification:1}] For any $a,b\in C$ let $a\star b \coloneq aba^{-1} = \phi_a\restr{C}\hspace{-0.8ex}(b)\in C$. It holds that $a\star a = aaa^{-1} = a$. The map $L_a = \phi_a\restr{C}$ is bijective because there is a well-defined map $\phi_{a^{-1}}\restr{C}:C\to C$ inverse to it. To finish, $L_a$ is also a~homomorphism:
            \begin{align*}
            L_a(b\star c) = \phi_a\restr{C}\hspace{-0.8ex}(bcb^{-1}) = \phi_a\restr{C}\hspace{-0.8ex}(b)\phi_a\restr{C}\hspace{-0.8ex}(c)\phi_a\restr{C}\hspace{-0.8ex}(b)^{-1} = L_a(b)L_a(c)L_a(b)^{-1} = L_a(b)\star L_a(c).
            \end{align*}
        \item[\Cref{it:classification:2}$\rimp$\Cref{it:classification:3}] Since $C$ is closed upon conjugation by elements of the group $\Gen{}{C}$, the set $C = \bigcup_{c\in C}\{hch^{-1}\mid h\in \Gen{}{C}\}$ is a union of conjugacy classes of $\Gen{}{C}$.
        \item[\Cref{it:classification:3}$\rimp$\Cref{it:classification:2}] Conjugacy classes of the group $\Gen{}{C}$ are closed upon conjugation by elements of $\Gen{}{C}$, and such a property is preserved upon unions.\qedhere
    \end{enumerate}
\end{proof}

By definition, we may also note that $C$ and $\Alg{G}$ define such a conjugation quandle uniquely. On the other hand, if $\Alg{C}$ is a conjugation quandle over a group $\Alg{G}$, then the same quandle is realized over the group $\Alg{H}\coloneq \Gen{}{C}$. \Cref{thm:classification} especially says that for a group $\Alg{G}$ and its element $e$, there is a conjugation quandle over $\Alg{G}$ on the conjugacy class $\cl{\Alg{G}}{e}\coloneq \{geg^{-1}\mid g\in G\}$. This quandle shall be denoted $\Cl{\Alg{G}}{e}$.

\cite[Proposition~3.2]{racks} gives a characterization of connectedness for \emph{faithful racks}, which are isomorphic to a special type of conjugation quandles. This was later also used and re-proven by Kayacan as Proposition~2 in~\cite{kayacan}. Here, we describe connectedness of conjugation quandles in general. 
\begin{theorem}\label{thm:connectedness-criterion}
    Let $\Alg{C}$ be a conjugation quandle over a group $\Alg{G}$, then TFAE
    \begin{enumerate}
        \item\label{it:connectedness-criterion:1} quandle $\Alg{C}$ is connected,
        \item\label{it:connectedness-criterion:2} $C$ is a conjugacy class of the group $\Gen{}{C}$.
    \end{enumerate}
\end{theorem}
\begin{proof}\hfill

    \begin{enumerate}
        \item[\Cref{it:connectedness-criterion:1}$\rimp$\Cref{it:connectedness-criterion:2}] We may pick $e\in C$. Due to \Cref{thm:classification}, $C$ is a union of conjugacy classes of $\Gen{}{C}$, and so $\cl{\Gen{}{C}}{e}\subseteq C$. To prove the missing inclusion $C \subseteq \cl{\Gen{}{C}}{e}$, let $c\in C$. Since $\Alg{C}$ is connected, there is $\varphi\in \LMlt{\Alg{C}}$ such that $c = \varphi(e)$. Using the equality $\LMlt{\Alg{C}} = \Inn{\Gen{}{C}}\restr{C}$, there is $h\in \Gen{}{C}$ such that $$c = \varphi(e) = \phi_h\restr{C}(e) \in \cl{\Gen{}{C}}{e}.$$
        \item[\Cref{it:connectedness-criterion:2}$\rimp$\Cref{it:connectedness-criterion:1}] Pick $e\in C$. Note that it is enough to show that for any $c\in C$, there is $\psi\in \LMlt{\Alg{C}}$ for which $c = \psi(e)$: then for any $a,b\in C$, there are $\psi_1, \psi_2\in\LMlt{\Alg{C}}$ such that $a = \psi_1(e)$ and $b = \psi_2(e)$; thus, $b = \psi_2\circ\psi_1^{-1}(a)$.

        For a given $c\in C = \cl{\Alg{\Gen{}{C}}}{e}$, there is $h\in \Gen{}{C}$ such that $c =\phi_h(e)$. Using generators, we may express $h$ in the form $\prod_{i=1}^nc_i^{s_i}$ for some $n\in \bb{N}, c_i\in C, s_i\in \{\pm 1\}$ to obtain
            \begin{align*}c = \phi_{\prod_{i=1}^nc_i^{s_i}}(e) = \phi_{{c_1}^{s_1}}\circ\cdots\circ\phi_{{c_n}^{s_n}}(e) = \phi_{{c_1}^{s_1}}\restr{C}\circ\cdots\circ\phi_{{c_n}^{s_n}}\restr{C}(e)= L_{c_1}^{s_1}\circ\cdots\circ L_{c_n}^{s_n}(e).&\qedhere
            \end{align*}
    \end{enumerate}
\end{proof}

The fact that a conjugacy class generates a normal subgroup yields the following corollary. Note that this corollary was already known---conjugacy class of a simple nonabelian group not only forms a conjugation quandle but also a faithful rack, for which the connectedness was described by \cite[Proposition~3.2]{racks} or~\cite[Proposition~2]{kayacan}. 
\begin{corollary}
    Given an element $e$ of a simple group $\Alg{G}$, the quandle $\Cl{G}{e}$ is connected.
\end{corollary}
\begin{proof}
    The subgroup $\Gen{}{\cl{\Alg{G}}{e}}$ is normal in $\Alg{G}$. If the quandle contains at least two elements, then $\Gen{}{\cl{\Alg{G}}{e}} = \Alg{G}$, and by \Cref{thm:connectedness-criterion}, it is connected. Otherwise, it has only one element, and so it is connected trivially.
\end{proof}

In the paper, quandles of the form $\Cl{\Alg{G}}{e}$ occur. Be warned that despite \Cref{thm:connectedness-criterion} being an equivalence, it is possible for such quandles to not be connected when the condition $\Alg{G} = \Gen{}{\cl{\Alg{G}}{e}}$ is not met. To ensure this is really the case, in \Cref{ex:connectedness} we show that the conjugacy class of $\Alg{G}$ may or may not split into more conjugacy classes of its subgroup $\Gen{}{\cl{\Alg{G}}{e}}$. Nevertheless, no matter the connectedness, in conjugation quandles representable as $\Cl{\Alg{G}}{e}$, all the left translations have the same cycle structure: for any $c,c'\in \cl{\Alg{G}}{e}$ there is $g\in G$ such that $c' = \phi_g(c)$. It can be verified that the map $\phi'\coloneq\phi_g\restr{\cl{\Alg{G}}{e}}$ is a bijection on $\cl{\Alg{G}}{e}$, and since $\phi'(a\star b) = \phi_g(aba^{-1})= \phi_g(a)\phi_g(b)\phi_g(a)^{-1} = \phi'(a)\star\phi'(b)$, it is an automorphism of $\Cl{\Alg{G}}{e}$. Thus, by \Cref{eq:translations-are-conjugate}, $L_{c'} = L_{\phi'(c)} = \phi'\circ L_c\circ (\phi')^{-1}$.

\begin{example}\label{ex:connectedness}
    Recall that in $\Sym{n}$, permutations are conjugate if and only if they have the same cycle structure, and that for $n\geq 3$, the alternating group $\Alt{n}$ is generated by $3$-cycles.
    \begin{itemize}
        \item The conjugation quandle $\Cl{\Sym{5}}{(1\ 2\ 3)}$ is connected because the set $\cl{\Sym{5}}{(1\ 2\ 3)} = \cl{\Alt{5}}{(1\ 2\ 3)}$ is a conjugacy class of $\Gen{}{\cl{\Sym{5}}{(1\ 2\ 3)}} = \Alt{5}$.
        \item The conjugation quandle $\Cl{\Sym{4}}{(1\ 2\ 3)}$ is not connected since the set $\cl{\Sym{4}}{(1\ 2\ 3)} = \cl{\Alt{4}}{(1\ 2\ 3)}\dot{\cup}\cl{\Alt{4}}{(1\ 3\ 2)}$ is not a conjugacy class of $\Gen{}{\cl{\Sym{4}}{(1\ 2\ 3)}} = \Alt{4}$.
    \end{itemize}
    We note that it is a famous fact that each conjugacy class $\cl{\Sym{n}}{e}$ with $e\in \alt{n}$ equals one conjugacy class of $\Alt{n}$ or splits into two distinct conjugacy classes of $\Alt{n}$. Moreover, there is a criterion that fully describes when this splitting happens.
\end{example}

\section{Regular Cycles of Left Translations}

In this section, considering conjugation quandles, we will study what it means for a left translation to have a regular cycle. We will be able to find a complete characterization based on the structure of the defining group, extending Kayacan's Proposition~5 from~\cite{kayacan}.

The following is a simple observation.
\begin{lemma}\label{lemma:regular-cycle}
    If $\Alg{Q}$ is a finite quandle and $\pi \in \LMlt{\Alg{Q}}$, then TFAE
    \begin{enumerate}
        \item\label{it:regular-cycle:1} permutation $\pi$ contains a regular cycle,
        \item\label{it:regular-cycle:2} there is $z\in Q$ such that the stabilizer ${\Gen{}{\pi}}_z$ is trivial.
    \end{enumerate}
\end{lemma}
\begin{proof}\hfill

    \begin{enumerate}
        \item[\Cref{it:regular-cycle:1}$\rimp$\Cref{it:regular-cycle:2}] If we choose $z\in Q$ to be from any of the longest (\emph{regular}) cycles of $\pi$, then $\pi^k(z) = z$ already implies $\pi^k = \operatorname{id}$.
        \item[\Cref{it:regular-cycle:2}$\rimp$\Cref{it:regular-cycle:1}] Let $\ell$ be the length of the cycle of $\pi$ containing such $z$. By \emph{(2)}, $\pi^\ell =\operatorname{id}$. Given any other cycle in $\pi$, we thus witness that its elements are fixed by $\pi^\ell$, and so its length has to divide $\ell$.\qedhere
    \end{enumerate}
\end{proof}

The following is a part of \cite[Lemma~1.9]{racks}, also used by Kayacan in~\cite{kayacan}.
\begin{lemma}\phantomsection\label{lemma:chain}
    Let $\Alg{C}$ be a conjugation quandle over a group $\Alg{G}$. If we denote $\Alg{H}\coloneq \Gen{}{C}$, then there is a group isomorphism\begin{gather*}\psi:\slant{\Alg{H}}{\operatorname{Z}(\Alg{H})}\to\Inn{\Alg{H}}\restr{C} = \LMlt{\Alg{C}},\\ h\operatorname{Z}(\Alg{H})\mapsto \phi_h\restr{C}.\end{gather*}
\end{lemma}
\begin{proof}
    The map $\psi':\Alg{H}\to\Inn{\Alg{H}}\restr{C}$ defined by $h\mapsto \phi_h\restr{C}$ is a surjective group homomorphism. It holds that $h\in \operatorname{Ker}(\psi')$ if and only if $h\in H$ commutes with every element of $C$. The latter happens if and only if it commutes with every element of $\Gen{}{C} = \Alg{H}$; thus, $\operatorname{Ker}(\psi') = \operatorname{Z}(\Alg{H})$. By the first isomorphism theorem, there is an induced isomorphism $\psi:\slant{\Alg{H}}{\operatorname{Z}(\Alg{H})}\to\Inn{\Alg{H}}\restr{C}$ such that $h\operatorname{Z}(\Alg{H})\mapsto \phi_h\restr{C}$.
\end{proof}

For finite conjugation quandles, we are now ready to extend Kayacan's method, namely, Proposition~5 from~\cite{kayacan}. In our phrasing, this would have originally worked only for conjugation quandles that generate groups with a trivial center. We use \Cref{lemma:chain} and extend it to work in general, giving us a characterization of the Hayashi property for conjugation quandles of the form $\Cl{\Alg{G}}{e}$ later on. In return, this will allow us to expose the possible group-theoretic consequences of quandle theory and motivate \Cref{problem}.
\begin{proposition}\label{prop:centralizer-property-general}
    For a finite conjugation quandle $\Alg{C}$ over a group $\Alg{G}$ and any $c\in C$, TFAE
    \begin{enumerate}
        \item $L_c$ in $\Alg{C}$ contains a regular cycle,
        \item there is $z\in C$ such that $\Gen{}{c}\cap\operatorname{C}(z) \leq \operatorname{Z}\left(\Gen{}{C}\right)$.
    \end{enumerate}
\end{proposition}
\begin{proof}
    Denote $\Alg{H}\coloneq \Gen{}{C}$. Let $\psi$ be the isomorphism from \Cref{lemma:chain}. For any $z\in C$, we have
    \begin{align*}
        \psi^{-1}\left({\Gen{}{L_c}}_z\right) = \psi^{-1}\left(\Gen{}{L_c}\cap \LMlt{\Alg{C}}_z\right)= \psi^{-1}\left(\Gen{}{L_c}\right) \cap \psi^{-1}\left(\LMlt{\Alg{C}}_z\right)\\= \Gen{}{\psi^{-1}(\phi_c\restr{C})}\cap \psi^{-1}\left(\left(\Inn{\Alg{H}}\restr{C}\right)_z\right) \\=\Gen{}{c\operatorname{Z}(\Alg{H})}\cap \left\{\psi^{-1}(\phi_h\restr{C})\mid h\in H, \phi_h\restr{C}(z) = z\right\}\\=  \left\{x\operatorname{Z}(\Alg{H})\mid x\in \Gen{}{c}\right\}\cap \left\{h\operatorname{Z}(\Alg{H})\mid h\in \Cent{\Alg{H}}{z}\right\}.
    \end{align*}

    Given $x\operatorname{Z}(\Alg{H}) = h\operatorname{Z}(\Alg{H})$, an element of this intersection, we may find an element $g\in \operatorname{Z}(\Alg{H})$ such that $h = xg$ and since $z\in C\subseteq H$, it holds that $$\phi_x(z) = xzx^{-1} = (xg)z(xg)^{-1} = \phi_h(z) = z.$$ Thus, $x\in\Cent{\Alg{H}}{z}$ and we managed to show that\footnote{the skipped inclusion is trivial}
    \begin{equation}\label{eq:preimage}
        \psi^{-1}\left({\Gen{}{L_c}}_z\right) = \left\{x\operatorname{Z}(\Alg{H})\mid x\in\Gen{}{c}\cap \Cent{\Alg{H}}{z} \right\}.
    \end{equation}

    By \Cref{lemma:regular-cycle}, $L_c$ contains a regular cycle if and only if there is $z\in C$ such that ${\Gen{}{L_c}}_z$ is trivial. This happens if and only if there is $z\in C$ such that $\psi^{-1}\left({\Gen{}{L_c}}_z\right)$ is trivial, i.e., using \Cref{eq:preimage}, $\Gen{}{c}\cap \Cent{\Alg{G}}{z} = \Gen{}{c}\cap \Cent{\Alg{H}}{z}\subseteq \operatorname{Z}(\Alg{H})$.
\end{proof}

\begin{corollary}\label{cor:regular-cycle-characterization-cl}
   Let $e$ be an element of a group $\Alg{G}$ such that $\Cl{\Alg{G}}{e}$ is finite. Then TFAE
    \begin{enumerate}
        \item\label{it:regular-cycle-characterization-cl:1} the left translation $L_e$ in $\Cl{\Alg{G}}{e}$ contains a regular cycle,
        \item\label{it:regular-cycle-characterization-cl:2} there exists $z\in \cl{\Alg{G}}{e}$ such that $\Gen{}{e}\cap\operatorname{C}(z) \leq \operatorname{Z}\left(\Gen{}{\cl{\Alg{G}}{e}}\right)$,
        \item\label{it:regular-cycle-characterization-cl:3} the conjugacy class $\cl{\Alg{G}}{e}$ is good in $\Alg{G}$,
        \item\label{it:regular-cycle-characterization-cl:4} the quandle $\Cl{\Alg{G}}{e}$ has the Hayashi property.
    \end{enumerate}
\end{corollary}
\begin{proof}
    We have observed in the last paragraph of \Cref{sec:conjugation-quandles} that \Cref{it:regular-cycle-characterization-cl:1}$\rlimp$\Cref{it:regular-cycle-characterization-cl:4}. The equivalence \Cref{it:regular-cycle-characterization-cl:1}$\rlimp$\Cref{it:regular-cycle-characterization-cl:2} is given directly by \Cref{prop:centralizer-property-general}. Now, note that for any $c,z\in \cl{\Alg{G}}{e}$, the statement $\Gen{}{c}\cap\operatorname{C}(z) \leq\operatorname{Z}\left(\Gen{}{\cl{\Alg{G}}{e}}\right)$ is equivalent to saying that, for each integer $k$, we have $c^k\in \operatorname{Z}\left(\Gen{}{\cl{\Alg{G}}{e}}\right)$ if and only if $c^k$ commutes with $z$. Therefore, \Cref{prop:centralizer-property-general} also gives the equivalence \Cref{it:regular-cycle-characterization-cl:3}$\rlimp$\Cref{it:regular-cycle-characterization-cl:4}.
\end{proof}

\begin{remark}
    The implication \Cref{it:regular-cycle-characterization-cl:2}$\rimp$\Cref{it:regular-cycle-characterization-cl:3} in \Cref{cor:regular-cycle-characterization-cl} can be proven by construction. Given a pair $e,z$ satisfying \Crefitem{cor:regular-cycle-characterization-cl}{it:regular-cycle-characterization-cl:2}, it can be shown that for any $g\in G$, the pair $\phi_g(e), \phi_g(z)$ also satisfies \Crefitem{cor:regular-cycle-characterization-cl}{it:regular-cycle-characterization-cl:2} by conjugating the inclusion $\Gen{}{e}\cap\operatorname{C}(z) \subseteq \operatorname{Z}\left(\Gen{}{\cl{\Alg{G}}{e}}\right)$.
\end{remark}

In particular, by \Cref{cor:regular-cycle-characterization-cl}, if $\Alg{G}$ is a good group, then every finite quandle of the form $\Cl{\Alg{G}}{e}$ has the Hayashi property. Let us now give the characterization of Hayashi's conjecture for conjugation quandles in this group-theoretic settings.
\begin{theorem}\label{thm:hayashi-equivalence}
    The following conditions are equivalent
    \begin{enumerate}
        \item\label{it:hayashi-equivalence:1} Hayashi's conjecture (\Cref{conj:hayashi}) holds for all finite connected conjugation quandles,
        \item\label{it:hayashi-equivalence:2} for every group $\Alg{G}$ and its finite conjugacy class $C$, if $\Gen{}{C} = \Alg{G}$, then $C$ is good in $\Alg{G}$.
    \end{enumerate}
\end{theorem}
\begin{proof}\hfill

    \begin{enumerate}
        \item[\Cref{it:hayashi-equivalence:1}$\rimp$\Cref{it:hayashi-equivalence:2}] Let $\Alg{G}$ be a group and $\cl{\Alg{G}}{e}$ a finite conjugacy class satisfying $\Gen{}{\cl{\Alg{G}}{e}} = \Alg{G}$. Using \Cref{thm:connectedness-criterion}, the conjugation quandle $\Cl{\Alg{G}}{e}$ is connected; thus, by \Cref{it:hayashi-equivalence:1}, it has the Hayashi property. \Cref{cor:regular-cycle-characterization-cl} applies and finishes the proof.
        \item[\Cref{it:hayashi-equivalence:2}$\rimp$\Cref{it:hayashi-equivalence:1}] Given a finite connected conjugation quandle $\Alg{C}$ over a group $\Alg{\tilde{G}}$, we denote $\Alg{G}\coloneq\Gen{}{C}$. By \Cref{thm:connectedness-criterion}, $C$ is a conjugacy class of $\Alg{G}$. Therefore, $C$ is good in $\Alg{G}$ by \Cref{it:hayashi-equivalence:2} and \Cref{cor:regular-cycle-characterization-cl} implies that $\Alg{C}$ has the Hayashi property. \qedhere
    \end{enumerate}
\end{proof}

\section{Special cases}\label{sec:special}
In this section, we apply the developed theory to show that particular families of groups are good. Using the correspondence given by \Cref{cor:regular-cycle-characterization-cl}, such types of results verify Hayashi's conjecture for the corresponding conjugation quandles whenever at least one of the conditions from \Cref{thm:connectedness-criterion} holds.

The following fact was observed by Kayacan who used it implicitly in the proof of Theorem~7 from \cite{kayacan}. It shows that all $p$-groups are good.
\begin{proposition}\phantomsection\label{prop:powers-of-prime}
    Let $\Alg{C}$ be a finite conjugation quandle over a group $\Alg{G}$. If the order of $c\in C$ in $\Alg{G}$ is a prime power (including $1$), then the left translation $L_c$ in $\Alg{C}$ contains a regular cycle.
\end{proposition}
\begin{proof}
    Let $c$ be of order $p^k$ for a prime number $p$ and $k\in\bb{N}_0$. Since $L_c^{(p^k)} = \phi_{c^{(p^k)}}\restr{C} = \phi_1\restr{C} = 1$, we know that the order of $L_c$ divides $p^k$; thus, cycle lengths of $L_c$ are powers of $p$ (including $1$) and all of them divide the largest one.
\end{proof}

We now show that the direct product of finite conjugation quandles is again a finite conjugation quandle when we consider it being over the direct product of the corresponding groups. This will let us show that finite nilpotent groups are good.
\begin{proposition}\phantomsection\label{prop:product}
    Let $\Alg{C}_1$ and $\Alg{C}_2$ be finite conjugation quandles. If $\Alg{C}_1$ and $\Alg{C}_2$ have the Hayashi property, then the direct product $\Alg{C}_1\times\Alg{C}_2$ has the Hayashi property. 
\end{proposition}
\begin{proof}
    Let $L_{(c_1,c_2)}$ be a left translation. In general, the length of the cycle of $(x_1,x_2)$ in $L_{c_1,c_2}$ equals the least common multiple of the lengths of the cycles of $x_1$ and $x_2$ in $L_{c_1}$ and $L_{c_2}$, respectively. Pick elements $z_1$ and $z_2$ from any of the longest (regular) cycles of $L_{c_1}$ and $L_{c_2}$, respectively. All the cycle lengths in $L_{(c_1,c_2)}$ divide the length of the cycle of $(z_1, z_2)$; thus, $L_{(c_1,c_2)}$ contains a regular cycle.
\end{proof}
\begin{corollary}\phantomsection\label{cor:finite-products}
    Finite direct product of good groups is a good group.
\end{corollary}
\begin{proof}
    Every conjugacy class in a direct product of two groups is a direct product of some of their conjugacy classes, and so the direct product of two good groups is good by \Cref{prop:product} and \Cref{cor:regular-cycle-characterization-cl}. Using induction, this extends to any finite direct products.
\end{proof}
\begin{corollary}
    Finite nilpotent groups are good.
\end{corollary}
\begin{proof}
    By \cite[Theorem~11]{nilpotent}, every finite nilpotent group is a finite direct product of $p$-groups. Every such $p$-group is good by \Cref{prop:powers-of-prime}, and so is their finite direct product according to \Cref{cor:finite-products}.
\end{proof}

Another family of groups we prove to be good are finite dihedral groups.
\begin{theorem}\label{thm:dihedral}
    For any $n\in \bb{N}$ such that $n\geq 3$, the dihedral group $\Dih{2n}$ of order $2n$ is good.
\end{theorem}
\begin{proof}
    Let $\cl{\Dih{2n}}{e}$ be a conjugacy class in $\Dih{2n}$. Using \Cref{cor:regular-cycle-characterization-cl}, to show that $\cl{\Dih{2n}}{e}$ is good, it is enough to show that $L_e$ contains a regular cycle.

    By the definition of $\Dih{2n}$, there are elements $s,o\in\dih{2n}$ such that \[\Dih{2n} = \Gen{}{s,o\mid s^2 = o^n = 1, sos^{-1} = o^{-1}}.\] Since $sos = s^{-1}os = s^{-1}os^{-1} = sos^{-1} = o^{-1}$, it holds that $\dih{2n}\subseteq \{so^k, o^k\mid k\in \bb{Z}\}$.
    
    \begin{itemize}
        \item If $e = so^k$, then $e^2 = (so^ks)o^k = o^{-k}o^k = 1$, and so the order of $e$ divides two. Thus, by \Cref{prop:powers-of-prime}, $L_e$ contains a regular cycle.
        \item If $e = o^k$, then $\cl{\Dih{2n}}{e}\subseteq \{\phi_{so^{k'}}(e), \phi_{o^{k'}}(e)\mid k'\in \bb{Z}\} = \{o^{\pm k}\}$. Thus, $L_e = 1$ because $L_e$ is a permutation on a set with at most two elements with a fixed point $L_e(e) = e$.\qedhere
    \end{itemize}   
\end{proof}

The groups $\Sym{n}$ and $\Alt{n}$ are also good. This would follow from the arguments made by Kayacan in the proof of Theorem~7 from \cite{kayacan}. This theorem states that connected conjugation quandles over the groups $\Sym{n}$ satisfy Hayashi's conjecture, i.e., have the Hayashi property. It was then explained how to adapt its proof for connected conjugation quandles over the groups $\Alt{n}$ as well.

Unfortunately, the presented proof of this theorem is not sound as it contains an invalid implication. Here, we first explain why the argument fails and, to reassure us that the groups $\Sym{n}, \Alt{n}$ are indeed good, we fill in this gap by proving \Cref{prop:z-in-sym-and-alt}.

Given a permutation $x\neq 1$ in a symmetric group $\Sym{n}$ for $n\geq 5$ such that the conjugation quandle $\Cl{\Sym{n}}{x}$ is connected, Kayacan uses \cite[Proposition~5]{kayacan} to argue that it is enough to find $y\in \cl{\Sym{n}}{x}$ so that no nontrivial power of $x$ centralizes $y$. Assuming that $x$ has at least two cycles of distinct lengths (referred to as ``case $r\geq 1$''), Kayacan notes that, clearly, there exists $y\in \cl{\Sym{n}}{x}$ such that: for every integer $s$ for which $x$ contains an $s$-cycle, the union of the supports of the $s$-cycles of $x$ is distinct from the union of the supports of the $s$-cycles of $y$. To finish the case $r\geq 1$, for any such $y$ and positive integer $k$ that divides $\ell :=\ord{\Sym{n}}{x}$, Kayacan then claims: ``Obviously, $x^k$ does not centralize $y$ unless $k = \ell$.''. In the following paragraph, we provide a nontrivial counterexample to the latter claim.

Let $x\coloneq (1\ 2)(3\ 4\ 5\ 6\ 7\ 8)$ and $y\coloneq (3\ 6)(1\ 4\ 5\ 2\ 7\ 8)$ be permutations from $\sym{8}$. The quandle $\Cl{\Alg{S}_8}{x}$ can be shown to be connected\footnote{As a permutation, $x$ is even, and so $\mathbf{H}\coloneq\Gen{\Alg{S}_8}{\cl{\Alg{S}_8}{x}}\leq \Alg{A}_8$. Moreover, $\Alg{H}\in\{\Alg{1}, \Alg{A}_8, \Alg{S}_8\}$ because it is a normal subgroup of $\mathbf{S}_8$. Therefore, $\Alg{H} = \Alg{A}_8$. There is a criterion that says that $\cl{\Alg{S}_8}{x} = \cl{\Alg{A}_8}{x}$, which proves the connectedness of $\Cl{\Alg{S}_8}{x}$ by \Cref{thm:connectedness-criterion}.}, the two presented permutations $x,y\in \cl{\Alg{S}_8}{x}$ have distinct sets of elements appearing in their 2-cycles and 6-cycles, respectively, and since $x^3 = (1\,2)(3\,6)(4\,7)(5\,8)=y^3$, contrary to Kayacan's claim, for $k = 3\neq 6 = \ord{\Alg{S}_8}{x} \eqcolon \ell$ we have $x^3y =y^4= yx^3$.

Therefore, the choice of the desired element $y$ cannot be made as arbitrary as presented in \cite[Theorem~7]{kayacan} and one needs to construct $y$ with greater diligence, making its existence unclear in return. To the best of our knowledge, the rest of the proof appears to be sound but it happens once again in the case $r=0$ that an explicit construction of the element $y$ is not presented.

To fill in these gaps, we prove the following statement. It gives an explicit construction for the element $y$ and finds it by conjugating the element $x$ by an element from $\alt{n}$. This makes the construction applicable for the proof of goodness for both the groups $\Sym{n}$ and $\Alt{n}$.
\begin{proposition}\label{prop:z-in-sym-and-alt}
    For any $n\geq 5$ and $e\in \sym{n}$, there is $z\in \{g eg^{-1}\mid g\in \alt{n}\}$ such that $\Gen{}{e}\cap\Cent{\Sym{n}}{z} = \Alg{1}$.
\end{proposition}

\begin{proof}
    Denote $(\lambda_{1}^{\alpha_1}, \lambda_{2}^{\alpha_2}, \dots, \lambda_{t}^{\alpha_t})$, the cycle structure of $e$. For notational reasons, when seen as a map, we relabel the domain of $e$ so that element $i$ is in a cycle of length $\lambda_i$ for all $i = 1, 2 \dots, t$. To prove the claim, we construct $z\in \{g eg^{-1}\mid g\in \alt{n}\}$ such that, for each integer $k$, \[e^k\in \Cent{\Sym{n}}{z}\rimp e^k = \operatorname{1}.\]

    \noindent\textbf{Case} $t\neq 1$. For $i = 1, 2, \dots, \lambda_1 - 1$ and $j = 1, 2, \dots, \lambda_t -1$, we put\footnote{It might happen that no $u_j$ is defined.} $u_i\coloneq e^i(1)$ and $x_j\coloneq e^j(t)$. By the assumption $t\neq1$, we know $1, \dots, t, u_1, \dots, u_{\lambda_1 -1}, x_1, \dots, x_{\lambda_t-1}\in \operatorname{Dom}(e)$ are pairwise distinct. Denote $T\coloneq \{1, 2, \dots, t\}$, $U\coloneq \{u_1, u_2, \dots, u_{\lambda_1 - 1}\}$, $X\coloneq \{x_1, x_2, \dots, x_{\lambda_t - 1}\}$.
    
    Since $\lambda_s \geq \lambda_1 + (s-1)$, for $s = t$ we obtain $\left|X\right|  = \lambda_t - 1\geq \lambda_1 + t - 2$, and so $x_1, x_2, \dots, x_{\lambda_1 + t - 2}\in X\subseteq \operatorname{Dom}(e)$. Therefore, we may define the permutation\footnote{If $\lambda_1 = 1$, then $\rho = \operatorname{1}$.} $$\rho \coloneq \prod_{i = 1, i\text{ odd}}^{\lambda_1 - 1}(x_{t-1 + i}\ u_i)\in \sym{n},$$
    and since $\left|X\right|\geq\lambda_1 + t - 2\geq t-1$, we may also define $$\pi\coloneq \prod_{j=1}^{t-1}(x_j\ t-j)\in \sym{n}.$$

    In the case $\lambda_t \geq 3$, we know $\left|X\right| = \lambda_t - 1  \geq2$, and so $x_1, x_2\in X\subseteq \operatorname{Dom}(e)\setminus \left(T\cup U\right)$. Otherwise, since $t\neq 1$, we get $\lambda_t = 2$. Then it has to be $t = 2, \lambda_1 = 1$ and since $\left|\operatorname{Dom}(e)\setminus\left(T\cup U\right)\right| =  n - (t + (\lambda_1 - 1)) = n - 2 \geq 3$, we may pick two distinct $y_1, y_2 \in \operatorname{Dom}(e)\setminus \left(T\cup U\right)$. Put $$\sigma \coloneq \begin{cases}
        \rho\pi,& \rho\pi \in \alt{n}\\
        (x_1\ x_2)\circ \rho\pi,& \rho\pi \notin \alt{n}, \lambda_t \geq 3\\
        (y_1\ y_2)\circ \rho\pi,& \text{otherwise.}
    \end{cases}$$
    By the definition, $\sigma \in \alt{n}$, and in all cases, for $i = 1, 2, \dots, t-1$ it holds $\sigma(x_{i}) = t-i$. We put $z\coloneq \sigma e\sigma^{-1}$. Computations show that
    \begin{align*}
        &z(i) = \sigma e \sigma^{-1}\left(i\right) = \sigma e(x_{t-i}) = \sigma(x_{t-i+1}) = i-1,& i = 2, 3, \dots, t-1\\
        &z(t) = \sigma e\sigma^{-1} (t) = \sigma e(t) = \sigma(x_1) = t-1.&
    \end{align*}

    Let there be $k\in \bb{Z}$ such that $e^k\in \Cent{\Sym{n}}{z}$. In the case $\lambda_i \mid k$ for each $i = 1, \dots, t$, we get $e^k = 1$. We will show that the latter is the only possible case. For a contradiction, let there be $i$ for which $\lambda_i\nmid k$ and we denote $i_{\min}$ the smallest one such. Then $i_{\min}$ is the smallest $i = 1, 2, \dots, t$, for which $e^k(i) \neq i$.
    
    \begin{itemize}
        \item If $i_{\min}\geq 2$, then 
        $$ze^k(i_{\min}) = e^kz(i_{\min}) = e^k(i_{\min} - 1) = i_{\min} - 1 = z(i_{\min}).$$ By the injectivity of $z$, this gives $e^k(i_{\min}) = i_{\min}$, a contradiction.
        \item Otherwise, $i_{\min} = 1$ and so $\lambda_1 = \lambda_{i_{\min}}\nmid k$, especially $\lambda_1 \geq 2$. Since $\lambda_1 -1 \geq1$, we may verify that $\sigma(x_{t}) = u_1$, and because $\lambda_t - 1 \geq \lambda_1 + t -2 \geq t$, we also know that $e(x_{t-1}) = x_t$. We obtain
        $$ze^k(1) = e^kz(1) = e^k\sigma e \sigma^{-1}(1) = e^k\sigma e(x_{t-1}) = 
            e^k\sigma (x_{t}) = e^k(u_1)\in \{1\}\cup U.$$

        By the definition of $U$, $e^k(1)\in\{1\}\cup U$. Since $e^k(1) = e^k(i_{\min}) \neq i_{\min} = 1$, there has to be $i\in \{1, 2, \dots, \lambda_1 - 1\}$ for which $e^k(1) = u_i$. We now contradict the derived consequence $ze^k(1)\in \{1\}\cup U$:
        \begin{align*}
            &i\text{ is odd}\rimp& ze^k(1) = z(u_i) = \sigma e(x_{t-1+i}) = \begin{cases}
                \sigma(t) = t,&t+i =  \lambda_t \\
                \sigma(x_{t+i}) = x_{t+i},&t+i\neq \lambda_t 
            \end{cases}\\
            &i\text{ is even}\rimp& ze^k(1) = z(u_i) = \sigma e (u_i) = \begin{cases}
                \sigma (1) \in X,&i=\lambda_1 - 1\\
                \sigma (u_{i+1}) = x_{t+i},&i\neq \lambda_1 -1
            \end{cases}
        \end{align*}
    \end{itemize}

   \noindent\textbf{Case} $t = 1$. In this case, $\operatorname{ord}_{\Sym{n}}(e) = \lambda_1$, and so $\Gen{}{e} = \{e^k\ |\ k = 0, 1, \dots, \lambda_1 - 1\}$. We only need to find $z\in \{g eg^{-1}\mid g\in \alt{n}\}$, such that $ze^k\neq e^kz$ for $k = 1, 2, \dots, \lambda_1 - 1$.
   \begin{itemize}
       \item If $\lambda_1 = 1$, then $e = 1$ and $\Gen{}{e} = \Alg{1}$, so the result is trivial.
       \item If $\lambda_1 = 2$, then, due to the assumption $n\geq5$, it holds that $\alpha_1 \geq3$. For $i = 1, 2, \dots, 5$ we may find $u_i\in \operatorname{Dom}(e)$ in such a way that $1, u_1, \dots, u_5$ are pairwise distinct and $e = (1\ u_1)(u_2\ u_3)(u_4\ u_5)\hat{e}$ for some $\hat{e}\in \sym{n}$ acting as an identity on $\{1, u_1, u_2, \dots, u_5\}$. Put $\sigma\coloneq (u_1\ u_2)(u_3\ u_4)\in \alt{n}$ and $z\coloneq \sigma e\sigma^{-1}$. It holds
       \begin{gather*}
         ze(1) = z(u_1) = \sigma e \sigma^{-1}(u_1) = \sigma e(u_2) = \sigma(u_3) = u_4,\\
         ez(1) = e\sigma e \sigma^{-1}(1) = e\sigma e (1) = e\sigma (u_1) = e(u_2) = u_3.
       \end{gather*}
       We proved that $ze^1(1) = u_4\neq u_3 = e^1z(1)$.
       \item If $\lambda_1 = 3$, then, due to the assumption $n\geq5$, it holds that $\alpha_1 \geq2$. For $i = 1, 2, \dots, 5$ we may find $u_i\in \operatorname{Dom}(e)$ in such a way that $1, u_1, \dots, u_5$ are pairwise distinct and $e = (1\ u_1\ u_2)(u_3\ u_4\ u_5)\hat{e}$ for some $\hat{e}\in \sym{n}$ acting as an identity on $\{1, u_1, \dots, u_5\}$. Put $\sigma\coloneq (u_2\ u_3\ u_4)\in \alt{n}$ and $z\coloneq \sigma e \sigma^{-1}$. It holds
       \begin{gather*}
           ze(1) = z(u_1) = \sigma e \sigma^{-1}(u_1) = \sigma e (u_1) = \sigma (u_2) = u_3,\\
           ze^2(1) = z(u_2) = \sigma e \sigma^{-1}(u_2) = \sigma e (u_4) = \sigma(u_5) = u_5,\\
           ez(1) = e\sigma e \sigma^{-1}(1) = e \sigma e (1) = e\sigma (u_1) = e(u_1) = u_2,\\
           e^2z(1) = e(ez(1)) = e(u_2) = 1.
       \end{gather*}
       We proved that $ze^1(1) = u_3 \neq u_2 = e^1z(1)$, $ze^2(1) = u_5 \neq 1 = e^2z(1)$.
       \item Otherwise, $\lambda_1\geq 4$. For $i = 0, 1, \dots, \lambda_1 -1$ denote\footnote{$u_0 = e^0(1) = 1$} $u_i \coloneq e^i(1)$. Elements $u_0, u_1, \dots, u_{\lambda_1 - 1}$ are pairwise distinct. Put $\sigma \coloneq (u_0\ u_1\ u_2)\in \alt{n}$ and $z\coloneq \sigma e\sigma^{-1}$. For any $k = 1, 2, \dots, \lambda_1 - 1$, it now holds $e^k(u_j) = u_{j+k \operatorname{mod} \lambda_1}$. That shows $e^kz(u_0) = e^k\sigma e \sigma^{-1}(u_0) = e^k\sigma e (u_2) = e^k\sigma (u_3) = e^k(u_3) = u_{3+k \operatorname{mod} \lambda_1}$. On the other hand,
       $$ze^k(u_0) = z(u_{k\operatorname{mod}\lambda_1}) = z(u_k) =  \begin{cases}
            u_2,& \text{if } k = 1,\\
            u_0,& \text{if } k = 2,\\
            u_{k+1},& \text{if } 3\leq k < \lambda_1 - 1,\\
            u_1,& \text{if } k= \lambda_1 - 1.
     \end{cases}$$
    
    For a contradiction, if it was the case that $e^kz(u_0) = ze^k(u_0)$, then
     $$3+k\operatorname{mod}\lambda_1 = \begin{cases}
            2,& \text{if } k = 1,\\
            0,& \text{if } k = 2,\\
            k+1,& \text{if } 3\leq k< \lambda_1 - 1,\\
            1,& \text{if } k = \lambda_1 - 1.
     \end{cases}$$
    Which may be reduced in modular arithmetic to
    $$0 \equiv_{\lambda_1} \begin{cases}
            -k-1\equiv_{\lambda_1} -2,& \text{if } k = 1,\\
            -k-3\equiv_{\lambda_1} -5,& \text{if } k =  2,\\
            -2,& \text{if } 3\leq k < \lambda_1 - 1,\\
            -k-2\equiv_{\lambda_1}-1,& \text{if } k = \lambda_1 - 1.
     \end{cases}$$
     Since $\lambda_1\geq4$, the only option left is $k=2, \lambda_1 = 5$. Nevertheless, the contradiction arises elsewhere: $e^2z(u_2) = ze^2(u_2) = z(u_4) =\sigma e \sigma^{-1}(u_4) = \sigma e (u_4) = \sigma (u_0) = u_1 \neq u_2 = e^2(u_0) = e^2z(u_2).$\qedhere
   \end{itemize}
\end{proof}

The following theorem now re-states the result of Kayacan \cite[Theorem~7]{kayacan} in the stronger form we were interested in.
\begin{theorem} \label{thm:sym-and-alt}
    For any $n\in \bb{N}$, the group $\Alt{n}$ (resp. $\Sym{n}$) is good.
\end{theorem}
\begin{proof}
    Let $\Alg{C}\coloneq \Cl{\Alt{n}}{e}$ for $e\in \alt{n}$ (resp. $\Alg{C}\coloneq \Cl{\Sym{n}}{e}$ for $e\in \sym{n}$) be a conjugation quandle.
    \begin{itemize}
        \item If $n\geq 5$, then by \Cref{prop:z-in-sym-and-alt} there is $z\in\left\{geg^{-1}\mid g\in\alt{n}\right\}\subseteq C$ such that $\Gen{}{e}\cap\Cent{\Sym{n}}{z} =\Alg{1}$. Thus, $C$ is good by \Cref{cor:regular-cycle-characterization-cl}.
        \item If $n\leq 4$, then the element $e$ has to have order of at most four in $\Alt{n}$ (resp. $\Sym{n}$). Each such admissible order is a prime power (possibly 1) and using \Cref{prop:powers-of-prime}, we deduce that $L_e$ contains a regular cycle. By \Cref{cor:regular-cycle-characterization-cl}, we deduce that $C$ is good.\qedhere
    \end{itemize}
\end{proof}

In \Cref{ex:connectedness}, we pointed out that there exist conjugation quandles of the form $\Cl{\Sym{n}}{e}$ that are not connected. This makes them not a subject of Hayashi's conjecture, and yet, since the groups $\Sym{n}$ are good, such quandles have the Hayashi property. Similarly, this applies to finite nilpotent and dihedral groups. Computationally, we have also verified that simple groups of order at most $5\,000\,000$ and general groups of order at most $500$ are good. This was done using GAP~\cite{gap, smallgrp, atlasrep} and is presented in~\Cref{sec:comp}. This phenomenon motivated us to state \Cref{problem}.

\section{Acknowledgement}
This article is based on the author's bachelor's thesis. The author would like to thank David Stanovský for supervising it. In addition, the author is thankful for the valuable comments and suggestions made by an anonymous referee. 

\printbibliography

\newpage
\appendix
\section{Small Groups Are Good}\label{sec:comp}
Here, we present how we used GAP (Version 4.16.1)~\cite{gap} with the packages SmallGrp (Version 1.7.0)~\cite{smallgrp} and AtlasRep (Version 2.1.12)~\cite{atlasrep} to computationally verify that all groups of order at most $500$ and all simple groups of order at most $5\,000\,000$ are good.

Let there be a finite group $\Alg{G}$. If $\Alg{G}$ is abelian, then it is good because every element $e\in G$ is central in every subgroup of $\Alg{G}$ that contains the element $e$.

If $\Alg{G}$ is nonabelian, then we want to verify that every conjugacy class \gap{cl} is good in $\Alg{G}$. We denote $\Alg{H}\coloneq\Gen{}{\gap{cl}}$. Arbitrarily picking a single representative $\gap{e}\in \gap{cl}$, it is by \Cref{cor:regular-cycle-characterization-cl} that $\gap{cl}$ is good in $\Alg{G}$ if and only if there exists $\gap{z}\in\gap{cl}$ such that $$\Gen{}{\gap{e}}\cap \Cent{\Alg{H}}{\gap{z}} = \Gen{}{\gap{e}}\cap \Cent{\Alg{G}}{\gap{z}}\subseteq \operatorname{Z}(\Alg{H}).$$
In GAP, the goodness of $\Alg{G}$ can thus be determined using the following function when $\Alg{G}$ is set as its input.
\begin{lstlisting}[style=gap]
IsGoodGeneral := function(G)
  local cl, H, Z, C;
  if IsAbelian(G) then 
    return true;
  fi;

  for cl in ConjugacyClasses(G) do
    H := Group(AsList(cl));;
    Z := Center(H);;
    C := Group([Representative(cl)]);;
    if not ForAny(cl,
        z -> IsSubset(Z, Intersection(C, Centralizer(H, z)))) then
      return false;
    fi;
  od;
  return true;
end;;
\end{lstlisting}

If we assume that the group $\Alg{G}$ is simple nonabelian, the computation can be simplified. Given a conjugacy class \gap{cl}, the group it generates is normal in $\Alg{G}$, and so we have $\Gen{}{\gap{cl}}\in \{\Alg{1}, \Alg{G}\}$. In both cases it holds that $\operatorname{Z}(\Gen{}{\gap{cl}})  = \Alg{1}$. Therefore, similarly to the previous paragraph, for $\Alg{H}\coloneq\Gen{}{\gap{cl}}$, it is now enough to decide for a single arbitrary representative $\gap{e}\in \gap{cl}$ whether there exists $\gap{z}\in \gap{cl}$ such that $\Gen{}{\gap{e}}\cap \Cent{\Alg{G}}{\gap{z}} = \Alg{1}$, which is equivalent to saying that the group $\Gen{}{\gap{e}}\cap \Cent{\Alg{G}}{\gap{z}}$ is of order $1$. This can be decided by the following function when $\Alg{G}$ is given as the input.
\begin{lstlisting}[style=gap]
IsGoodSimpleNonAbelian := function(G)
  local cl, C;

  for cl in ConjugacyClasses(G) do
    C := Group(Representative(cl));;
    if not ForAny(cl,
    \end{lstlisting}
    \begin{lstlisting}[style=gap]
        z -> Size(Intersection(C, Centralizer(G, z))) = 1) then
      return false;
    fi;    
  od;
  return true;
end;;
\end{lstlisting}

The two presented definitions of the functions we saved into a file named \gap{goodness-functions.g}.

The code to verify the goodness of general groups of order at most $500$ and simple nonabelian groups of order at most $5\,000\,000$ we saved into separate files \gap{computation-general.g} and \gap{computation-simple.g}, respectively.

For the general groups, the file went as follows.
\begin{lstlisting}[style=gap]
start := Runtime();;
nrTested := 0;;
nonGood := [];;

for n in [1..500] do
  for i in [1..NrSmallGroups(n)] do
    nrTested := nrTested + 1;
    if not IsGoodGeneral(SmallGroup(n, i)) then
      Add(nonGood, [n,i]);
    fi;
  od;
od;

Print("Computation CPU runtime: ", Runtime() - start, " ms\n");
Print("Number of groups tested: ", nrTested, "\n");
Print("Nongood groups found: ", nonGood, "\n");
\end{lstlisting}

For the simple nonabelian groups, the file's contents was as given below.
\begin{lstlisting}[style=gap]
start := Runtime();;
nrTested := 0;;
nonGood := [];;

for G in SimpleGroupsIterator(1, 5000000) do
  if not IsAbelian(G) then
    nrTested := nrTested + 1;
    if not IsGoodSimpleNonAbelian(G) then
      Add(nonGood, [Size(G), StructureDescription(G)]);
    fi;
  fi;
od;

Print("Computation CPU runtime: ", Runtime() - start, " ms\n");
Print("Number of groups tested: ", nrTested, "\n");
Print("Nongood groups found: ", nonGood, "\n");
\end{lstlisting}

To check that the specified package versions were used, we created one more file which we call \gap{setup.g}.
\begin{lstlisting}[style=gap]
if LoadPackage("smallgrp") <> true then
    Error("The SmallGrp package could not be loaded");
fi;

if LoadPackage("atlasrep") <> true then
    Error("The AtlasRep package could not be loaded");
fi;

Print("GAP version: ", GAPInfo.Version, "\n");
Print("SmallGrp version: ", PackageInfo("smallgrp")[1].Version, "\n");
Print("AtlasRep version: ", PackageInfo("atlasrep")[1].Version, "\n");
\end{lstlisting}

With all four files prepared, we ran a fresh GAP session, which turned out as we now present.
\begin{lstlisting}[style=gap]
gap> Read("setup.g");
GAP version: 4.16.1
SmallGrp version: 1.7.0
AtlasRep version: 2.1.12
gap> Read("goodness-functions.g");
gap> Read("computation-general.g");
Computation CPU runtime: 3102133 ms
Number of groups tested: 92570
Nongood groups found: [  ]
gap> Read("computation-simple.g");
Computation CPU runtime: 2598477 ms
Number of groups tested: 81
Nongood groups found: [  ]
\end{lstlisting}

Such computations were executed three times, ending up with similar runtimes and equal results. We used a 13-inch MacBook Air (M4, 2025) with a 10-core Apple M4 CPU, ARM64 architecture, and 16 GB of unified memory, running macOS 26.6.2. For details on how to install and operate GAP, we refer the reader to the official webpage \href{https://www.gap-system.org/}{\url{https://www.gap-system.org/}}.
\end{document}